\documentclass[psamsfonts]{amsart}
\usepackage{times,amsthm,amsfonts,amssymb,amscd,url}
\usepackage{tikz}
\usepackage{tikz-cd}

\theoremstyle{plain} 
\newtheorem{lemma}{Lemma}[section]
\newtheorem{proposition}[lemma]{Proposition}

\newtheorem*{mainproposition*}{Main Proposition}
\theoremstyle{definition}
\newtheorem*{hypotheses*}{Hypotheses for the Main Proposition}
\newtheorem*{question*}{Question}

\newtheorem{definition}[lemma]{Definition}
\newtheorem*{definition*}{Definition}
\newtheorem*{construction}{Construction}

\theoremstyle{remark} 
\newtheorem{remark}[lemma]{Remark} 
\newtheorem*{remark*}{Remark}

\newcommand{\id}{\operatorname{id}}

\newcommand{\End}{\operatorname{End}}

\newcommand{\R}{\mathbb{R}}
\newcommand{\C}{\mathbb{C}}

\newcommand{\E}{e}

\title{Cones in homotopy probability theory}

\author{Gabriel C. Drummond-Cole}
\address{Center for Geometry and Physics, 
Institute for Basic Science (IBS), Pohang 790-784, Republic of Korea}
\email{gabriel@ibs.re.kr}
\thanks{This work was supported by IBS-R003-G1}

\author{John Terilla}
\address{Department of Mathematics, The Graduate Center and Queens
  College, The City University of New
  York, USA}
\email{jterilla@gc.cuny.edu}



\begin{document}

\begin{abstract}
This note defines cones in homotopy probability theory
and demonstrates that a cone over a space is a reasonable 
replacement for the space.  
The homotopy Gaussian distribution in one variable is revisited as a
cone on the ordinary Gaussian.
\end{abstract}

 \maketitle

\section{Introduction}
This note is concerned with a simple categorical aspect of homotopy
probability theory.  Ordinary probability theory is concerned with a
vector space of random variables together with a linear
functional called the expectation.  The space of random
variables is also endowed with an associative product.  The
expectation does not respect the associative product.  Rather, the
failure of the expectation to be an algebra map defines correlations
among the random variables.

Homotopy probability theory~\cite{HPT1, HPT2} is a theory in which the space of random
variables is a chain complex and the expectation map is a chain map to
the ground field.  In homotopy probability theory, the 
chain complex of random variables is also endowed with an associative
product.  Just as in ordinary probability theory, the expectation map
is not assumed to respect the associative product of random variables.
The differential is also not assumed to respect the product structure.   
Homotopy probability theory generalizes ordinary probability
theory---every ordinary probability space
is a homotopy probability space where the space of random variables is a
chain complex concentrated in degree zero with zero differential. 
The article \cite{HPT2} explains how to obtain meaningful, homotopy
invariant correlations among the random variables by adapting the failure of the
expectation to be an algebra map, involving the failure of the differential
to be a derivation.


In this paper, cones on homotopy
probability spaces are introduced.  A cone is a factorization of the expectation map into
an inclusion followed by a quasi-isomorphism, the result of which is
that all of the relations among
expectation values become encoded 
in a differential.   
A cone on an ordinary probability
space is a homotopy probability space that serves as a good replacement for the original
probability space and in the cone, homological methods can be used to compute expectation
values and correlations of the original probability space.  In a variation where the expectation is
factored as an inclusion that is an algebra map followed by a
quasi-isomorphism, 
computations in the cone are more simply related to
those in the original probability space and can be easier to carry
out.  One source of these \emph{algebraic} cones is when the relations
among expectations arise from a group acting on the random variables.

\thanks{The authors would like to thank 
Jeehoon Park and Jae-Suk Park for explaining 
their paper \cite{PP} and to thank 
Owen Gwilliam and Dennis Sullivan for other helpful conversations.
}

\section{Preliminaries}
In ordinary probability theory where random variables are measurable
functions on a measure space and expectation is integration, one has the special random variable $1$
which serves a multiplicative identity for the product of random
variables.   The only compatibility between the algebra structure in
the space of random variables and the expectation is the normalization
that expectation value of the random variable $1$ equals the complex
number $1$.  In order to capture this normalization in homotopy probability
theory, pointed chain complexes
will be used.  It is worth noting that this normalization effectively
``cancels'' the trancendental unknowns involved in integration.

\begin{definition}

A \emph{pointed chain complex} is a chain complex $(V,d)$ together
with a map $\upsilon:\C \to V$ called the \emph{unit}.  
A morphism of pointed chain complexes is a
chain map commuting with the units.  The map $\upsilon$ can be
conveniently identified with the element $\upsilon(1)\in V$.
\end{definition}

The field of complex numbers $\C$ is considered as a chain complex concentrated in
degree zero with zero differential.  
  The identity $\id:\C \to \C$ makes $\C$ a pointed
chain complex.  

Definitions~\ref{def:probspace}~through~\ref{def:joint_cumulant} are taken with slight modification
from~\cite{HPT2}. The interested reader can find more details there.

\begin{definition}
\label{def:probspace}
A \emph{unital commutative homotopy probability space} is a pointed chain complex
$V$ together with the following data:
\begin{itemize}
\item A graded commutative, associative product on $V$.  The unit
  $\upsilon:\C \to V$ of the chain complex is the identity for the product.  No compatibility is assumed between the product
  and the differential.
\item A map of pointed chain complexes $\E:V\to \C$ called
  \emph{the expectation}.   
\end{itemize}
A morphism of unital commutative homotopy probability spaces is simply
a morphism of pointed chain complexes commuting with expectation.  No
compatibility is assumed between the morphism and the products.
A unital commutative homotopy probability space will be
called a \emph{probability space} for short.  If the underlying pointed
chain complex is concentrated in degree zero and has zero
differential, we will call it an \emph{ordinary probability space}.  
\end{definition}

Note that $\C$ becomes an ordinary probability space by defining the expectation
to be the identity map.  The assumption that expectation is a map of
pointed chain complexes means that $\E\upsilon=\id_\C$.  Therefore, both the
unit of a probability space $\upsilon:\C \to V$ and the expectation 
$\E:V \to \C$ are morphisms of probability
spaces and the unital conditions on morphisms imply that any map between $\C$ and a probability space must be one of these two.  Therefore, $\C$ is both initial and terminal in the category
of probability spaces.

In order to define correlations among random variables in a
probability space, the failures of (1) the expectation to be an algebra
map and (2) the 
differential to be a derivation must be taken into account.  A convenient way to organize the failure
to be a derivation is with the language of $L_\infty$ algebras.
Most concepts and calculations for an $L_\infty$ algebra $V$ can be
expressed  in terms of the symmetric (co)algebra of $V$.

\begin{definition}
Let $V$ be a graded vector space.  
We denote the cofree
nilpotent commutative coalgebra on $V$ by $SV$; it is linearly spanned by
symmetric powers of $V$.  

Given a symmetric associative product on $V$, 
define a map $\varphi$ from $SV$ to $V$ as follows:
\[
a_1\odot\cdots\odot a_n\mapsto a_1\cdots a_n.\]

The map $\varphi$ uniquely extends to be a coalgebra automorphism
$SV\to SV$.  By
abuse of notation, this coalgebra automorphism will also be denoted
$\varphi$.  In \cite{DS,RS}, $\varphi$ is called
\emph{the cumulant map}.
\end{definition}
Most concepts and calculations for an $L_\infty$ algebra $V$ can be
transported by $\varphi$.
\begin{definition}
Let $V$ be a graded vector space with an associative product. Let $f$ be a $\C$-linear map $SV\to
SV$. We call the compostion $\varphi^{-1}f\varphi$ the map $f$ {\em
  transported by $\varphi$}, and denote it $f^\varphi$.   If $V$ and $W$
are two graded vector spaces with associative products, then a
$\C$-linear map $SV\to SW$ can also be transported.  In this case,
$f^\varphi=\varphi^{-1}f \varphi$ where the $\varphi$ on the right is a coalgebra
automorphism $SV
\to SV$ and the $\varphi^{-1}$ on the left is the inverse of the coalgebra automorphism
$\varphi:SW \to SW$.  
\end{definition}\label{def:transported_structure}
Let $V$ be a probability space.   The differential $d$ on $V$ can
be extended to a square zero coderivation on $SV$.  This makes the
probability space into an $L_\infty$ algebra.  Transporting this
$L_\infty$ structure by $\varphi$ defines another $L_\infty$ structure
$d^\varphi$ on $V$
called {\em the transported structure}.

\begin{definition}
Let $V$ be a probability space. 
A {\em collection of $n$ homotopy random variables} is  an  $L_\infty$
morphism from $(\C^{\times n},0)$ to $(V,d^\varphi)$.  That is, it is a
degree zero map 
$X:S\C^{\times n}\to SV$ satisfying $d^\varphi X=0$.
\end{definition}
\begin{remark}
In \cite{HPT2}, a collection of homotopy random variables was defined
as the homotopy class of such a morphism, rather than a single
morphism. 
This definition is less obfuscatory.
\end{remark}
The expectation $\E$ can be viewed as an $L_\infty$ morphism from
$(V,d)$ to $(\C,0)$. This map can be transported by $\varphi$.
\begin{definition}
The {\em total cumulant} $K$ of a probability space is the
expectation map transported by $\varphi$: 
 \[K:= e^\varphi\]
\end{definition}

To summarize: the expectation is a map $\E:V \to \C$ satisfying $\E d=0$.
Transporting the expectation map by $\varphi$ results in the total
cumulant, which is a coalgebra map
$K:SV \to S\C$ satisfying $Kd^\varphi=0$.  A 
coalgebra map $K:SV \to S\C$ 
is completely determined by its components, which are
multilinear maps $\{k_n:V^{\times n} \to \C\}$.  
If $V$ is an ordinary probability space, then these multilinear maps
$\{k_n\}$ coincide precisely with the classical cumulants
in ordinary probability theory \cite{NSp}.  

\begin{definition}\label{def:joint_cumulant}
The {\em joint cumulant} of a collection of $n$ homotopy random variables
denotes the composition of the total cumulant map with the collection of
$n$ homotopy random variables.
\[(\C^{\times n},0) \stackrel{X}{\longrightarrow} (V,d) \stackrel{K}{\longrightarrow} (\C,0)\]
\end{definition}

Recall that two $L_\infty$ morphisms $X,Y:SV\to SW$ are homotopic if
there is an $L_\infty$ morphism $H:SV \to SW\otimes \C[t,dt]$ with
$H(0)=X$ and $H(1)=Y$.  
\begin{proposition}\label{prop1}
If two collections of
$n$ homotopy random variables are homotopic then they have identical
joint cumulants.
\end{proposition}
\begin{proof}See the proof of Lemma 3 in \cite{HPT2}.
\end{proof}

\section{Cones and algebraic cones}
\subsection{Contractible probability spaces}
Here, we identify a condition under which the converse of
Proposition \ref{prop1} is true.

\begin{definition}
A probability space is called {\em contractible} if the
expectation map is a quasi-isomorphism.
\end{definition}

\begin{proposition}\label{prop_contractible}
Let $V$ be a contractible probability space. Two collections
of $n$ homotopy random variables in $V$ are equal if and only if their
cumulants are
equal. 
\end{proposition}
\begin{proof}
Since $V$ is contractible, the total cumulant is a quasi-isomorphism
(since it is the transport of a quasi-isomorphism).
Since $K$ is a quasi-isomorphism, two collections of homotopy random variables are
homotopic if and only if their compositions with $K$ are homotopic.
Since $\C$ has no differential, this is true
if and only if their joint cumulants are equal.
\end{proof}

\begin{definition}Let $V$ be a probability space.  
  A \emph{cone on $V$} is a factorization $V\to CV\rightarrow\C$ of
  the expectation $V \to \C$ for
  which $CV$ is contractible and $V\to CV$ is an injective map of probability spaces.
\end{definition}

One can always factor a chain map $U \to V$ as an
inclusion followed by a surjective quasi-isomorphism
$U \rightarrowtail V \overset{\sim}{\twoheadrightarrow} W.$
In particular, the expectation of a probability space $e:V \to \C$ can 
be factored in this way, so cones on a probability space always exist.

\subsection{Algebra preserving morphisms}
While morphisms of probability spaces preserve expectation values,
they must be transported in order to relate joint cumulants.   
 
Let $\alpha:V \to
W$  be a morphism of probability spaces.   
One can relate the joint cumulants by a transported $\alpha$ as follows.
Suppose $X:(\C^{\times n},0)\to (V,d^\varphi)$ is a
collection of homotopy random variables and $K_V:(V,d^\varphi)\to
(\C,0)$ and $K_W:(W,d^\varphi)\to
(\C,0)$ are the total cumulants of $V$ and $W$ respectively, then the
joint cumulants satisfy $K_VX=K_W\alpha^\varphi X.$  The
following diagram illustrates the relationship between $\alpha$ and
the joint cumulants.  

  \begin{center}
 \begin{tikzcd}
 \C^{\times n}\ar{r}{X}
 &V
 \ar[bend left = 40]{rr}{K_V}
 \ar{d}{\varphi}
 \ar{r}[swap]{\alpha^\varphi}
 &W
 \ar{r}[swap]{K_W}
 \ar{d}{\varphi}
 &\C\ar{d}{\varphi}
 \\
 &V
 \ar{r}{\alpha}
 \ar[bend right = 40]{rr}[swap]{\E_V}
 &W
 \ar{r}{\E_W}
 &\C
 \end{tikzcd}
 \end{center}


\begin{proposition}\label{prop_algebramap}
Morphisms of probability spaces which preserve the algebra structure
  preserve total cumulants.
\end{proposition}
\begin{proof}Let $\alpha:V\to W$ be a morphism of probability spaces.
  Then $K_W \alpha^\phi =K_V$.  If the map $\alpha$ preserves the algebra
  structure then $\alpha^\phi=\alpha$.  So, $K_W \alpha =K_V$.
\end{proof}


\begin{definition}Let $V\rightarrow\C$ be a probability space.  An
  \emph{algebraic cone on $V$} is a cone $V\to CV\rightarrow\C$ on $V$
  for which $V\to CV$ is an algebra map.
\end{definition}

\begin{lemma}\label{lemma:coneexists}
  There exists an algebraic cone on any probability space.
\end{lemma}
\begin{proof}
  Decompose $V$ linearly into $H\oplus B\oplus \hat{B}$, where $H$ is
  a space of homology representatives including $1$, $B$ is the image
  of $d$, and $\hat{B}$ is a space of coimages of $d$ so that $d$ is
  an isomorphism from $\hat{B}$ to $B$. Let $K$ be the kernel of the
  expectation restricted to $H$; then $H$ is linearly spanned by $K$
  and $1$. Define $CV$ to be $V\oplus K[1]$ where $K[1]$ is a shifted
  copy of $K$. Extend $d$ by sending $K[1]$ to $K$ by the degree
  shift. Let the product of $K[1]$ with anything in $K[1]\oplus
  B\oplus \hat{B}\oplus K$ to be zero. Define the expectation of
  $K[1]$ to be zero. The evident inclusion from $V$ to $CV$ is a
  morphism of probability spaces which respects the product structure.
\end{proof}

\begin{remark}\label{classicalremark}
The fact that there exists an algebraic cone on any probability space 
allows an ordinary probability space to be replaced by homotopy
probability space with nice properties.  If $V\to \C$ is an ordinary
probability space, then for any collection of elements $x_1, \ldots,
x_n\in V$, the map $X:\C^n \to V$ defined by $e_i\mapsto x_i$ where $e_i$
is the $i$-th standard basis vector of $\C^n$ defines a collection of
homotopy random variables.   The cumulant of $X$ is an $L_\infty$
morphism $KX:\C^n \to \C$, whose value  on $e_{i_1}\odot
\cdots \odot e_{i_k}\in S^k\C^n$ precisely equals the classically defined cumulant 
\[KX_k \left (e_{i_1}\odot
\cdots \odot e_{i_k}\right)=c_k\left(x_{i_1}, \ldots, x_{i_n}\right).\]
Now let $V\to CV \to \C$ be an algebraic cone on $V$.
Because the first map is an algebra map, Proposition
\ref{prop_algebramap} applies and the cumulant of $X$ in $V$ equals
the cumulant of $\tilde{X}:\C^n \to V \to CV$.  Thus, the classical
cumulants can be computed within the cone $CV$.  Proposition \ref{prop_contractible}
applies to the second map of $V\to CV \to \C$ so the only
 homotopy random variables that have the same
cumulants are in fact homotopic in $CV$---all the possible
relations among the expectations of random variables in $V$ have been
encoded in the differential in $CV$.  Moreover, for simple collections
of homotopy random variables like $\tilde{X}$, whose components $S^k
\C^n \to CV$ are zero for $k>1$, many homotopy algebra 
computatations can be reduced to
simpler homology calculations in $CV$.
\end{remark}

\begin{remark}
The construction of an algebraic cone in the proof of 
Lemma \ref{lemma:coneexists}, which says an algebraic cone exists for any
probability space,
should be thought of an existence argument rather than a
construction.  The proof uses the kernel of the expectation but in
practice, one may not know much about this kernel and therefore may not have a description
of the algebraic cone constructed in the proof of Lemma
\ref{lemma:coneexists} that is explicit enough for calculations.  
Instead, one might find an algebraic cone by other means, say via
a relevant group action as described in the next section, and then the homotopy-algebra tools
available may be used to carry out calculations that otherwise involve
transcendental methods.
\end{remark}


\section{Group actions}
In Section 2 of \cite{PP}, there is a construction of a homotopy
probability space from a group acting on an ordinary probability space with a group-invariant
expectation.  The relations among the expectation values that arise
from the group action become encoded in the differential. 
In the special case
that \emph{all} relations among expecation values arise from the group
action, and assuming an additional vanishing condition,
this construction yields an algebraic cone.  

Suppose that $G$ is a 
Lie group acting on an ordinary probability space $V$ and that
the expectation map $e$ is $G$-equivariant.  That is, $e(gx)=e(x)$ for
all $g\in G$.  
In good cases (say $G$ is compact and simply
connected) the associated Lie
algebra action $\mathfrak{g}\to \End(V)$ determines the action of
$G$.  If $G$ captures all relations among expectations: $e(x)=e(y)$ if and only if 
 $y=gx$ for some $g\in G$, then all relations among expectations will
 be encoded in the $\mathfrak{g}$ action which satisfies $e(\lambda  x)=0$ for all 
$\lambda \in \mathfrak{g}$.

In this situation, $C(\mathfrak{g},V)$, the Chevalley-Eilenberg
cochain complex with values in the module $V$, produces a homotopy
probability space that with a suitable vanishing condition on higher
cohomology is an algebraic cone.  This cochain complex is
a non-positively graded complex \[\cdots \to \mathfrak{g}^*\otimes V \to
V \to 0\]with $V$ in degree zero. Further, the degree zero cohomology
is precisely the co-invariants $V/V_\mathfrak{g}$---an element
of $V$ is in the image of the differential if and only if it has the
form $\lambda x$ for some $\lambda\in \mathfrak{g}$.
The complex $C(\mathfrak{g},V)$ and
the differential can be described explictly as follows.
Let $\{\lambda_1,
\ldots, \lambda_n\}$ be
a basis for $\mathfrak{g}$ and let $\rho_i:V \to V$ denote the action of
$\lambda_i$ on $V$.
Let $\mathfrak{g}[-1]$ denote $\mathfrak{g}$ with its degree shifted by
$1$, so that an element $\lambda \in \mathfrak{g}[-1]$ has degree $1$
and an element $\eta\in (\mathfrak{g}[-1])^*$
has degree $-1$.
$$C(\mathfrak{g},V)=\hom(S \left(\mathfrak{g}[-1]\right),V)\simeq S
(\mathfrak{g}[-1])^*\otimes V.$$
Define a differential $d:CV \to CV$ by 
\begin{equation}\label{CEdifferential}
d=\sum_{i=1}^n \frac{\partial }{\partial \eta_i} \otimes \rho_i +
\sum_{i<j=1}^n \sum_{k=1}^n f_{ij}^k \eta_k \frac{\partial^2}{\partial \eta_i \partial
  \eta_j}\otimes 1
\end{equation}
where $\{\eta_1, \ldots, \eta_n\}$ is the basis for $\mathfrak{g}^*[1]$
dual to $\{\lambda_1, \ldots, \lambda_n\}$ and the $f_{ij}^k$ are the
structure constants $[\lambda_i,
\lambda_j]=\sum_{k=1}^n f_{ij}^k \lambda_k$.  

Since $H^0(CV,d)=V/V_\mathfrak{g}$ and the expectation is invariant, extending the
expectation $\E:CV \to \C$ to be zero on $S
(\mathfrak{g}[-1])^*$ is a chain map.  Moreover, since it is assumed
that $\E(x)=0$ if and only if $x\in V_\mathfrak{g}$, the expectation
$CV\to \C$ induces an isomorphism in $H^0$.
If, in addition, the cohomology $H^i(CV,d)$ vanishes in all other
degrees, the expectation $CV\to \C$ is a quasi-isomorphism.  This makes $V\to CV\to \C$
a cone on $V$, in fact an algebraic cone since $CV$ is a free
extension of $V$.

\section{An explicit homotopy in the Gaussian example}
In~\cite{HPT2}, the following cone on the ordinary probability space
for the one variable Gaussian was constructed.  See also section 3.2
of \cite{PP}, section 2.3
of \cite{G} and references therein.
\begin{definition}
Let $V=\C[x,\eta]$ where $x$ is in degree zero and $\eta$ is in degree $-1$ (so in particular, $\eta^2=0$). Define the expectation as 
\[
\E(p(x)+q(x)\eta)=\frac{\int_{-\infty}^\infty p(x)e^{\frac{-x^2}{2}}}{\int_{-\infty}^\infty e^{\frac{-x^2}{2}}}
\]
and the differential as
\[
d(p(x)+q(x)\eta)=q'(x)-xq.\]
We call this probability space the {\em homotopy Gaussian}.
\end{definition}
\begin{remark}
The one dimensional abelian Lie group $\R$ acts on $\R$ by translations $x
\mapsto x+g$ and induces an action on the space of functions
integrable with respect to the measure $e^{-\frac{x^2}{2}}dx$
by \[f(x)
\mapsto f(x+g)\exp \left( - \frac{g^2+2xg}{2}\right )\]
which can be seen to be integration-invariant by a simple change of variables.
The induced action
of the one dimensional abelian Lie algebra $\R$ is generated by the single map
$f(x)\mapsto f'(x)-xf(x)$.  Note that the Lie algebra action can be restricted to
the subspace of polynomials.  The Chevalley-Eilenberg complex with
values in the space of polynomials is precisely the homotopy
Gaussian.  The term $\sum f_{ij}^k \eta_k \frac{\partial^2}{\partial \eta_i \partial
  \eta_j}\otimes 1$ in the differential in Equation
\eqref{CEdifferential} vanishes since here the Lie algebra is
abelian.

Also, as with any abelian Lie algebra, the higher Chevalley-Eilenberg
cohomology of the homotopy Gaussian vanishes making it an algebraic
cone on the ordinary Gaussian (the degree zero part of the homotopy Gaussian).
\end{remark}
For the remainder of this section, assume that $X$ and $Y$ are two
ordinary random variables in the homotopy Gaussian; that is, $X(1)=p(x)$ and $Y(1)=q(x)$ with all higher terms zero. We will construct an explicit homotopy between $X$ and $Y$ in the case that they have the same cumulants $\kappa_i:S^i\C\to \C$. The reader may generalize to homotopy random variables with higher terms and to collections of homotopy random variables.
\begin{lemma}\label{lemma:1}
In the homotopy Gaussian, define 
\[y_n = -\eta\sum_{j=1}^{\lfloor\frac{n}{2}\rfloor}\frac{x^{n+1-2j}(n-1)!!}{(n+1-2j)!!}.
\]
Then $dy_n = x^n-E(x^n)$. 
\end{lemma}
\begin{proof}
The proof is a direct computation which is inductive.
\end{proof}
Note that $y_n$ as $n\ge 1$ varies span $\eta\C[x]$, so that the collection of $dy_n$ span the image of $d$.
\begin{lemma}
Assume $X$ and $Y$ have the same cumulants $\kappa_i$ for $i\le n$. Then $\E(p^n)=\E(q^n)$.
\end{lemma}
\begin{proof}
Examine the composition $\varphi\kappa$. The hypotheses of the lemma imply that $(\varphi\kappa X)_i=(\varphi\kappa Y)_i$ for $i\le n$. Since $\varphi\kappa=E\varphi$, this implies that $(E\varphi X)_n=(E\varphi Y)_n$. These maps take $(1\odot\cdots \odot 1)$ to $E(p^n)$ and $E(q^n)$, respectively.
\end{proof}
\begin{lemma}\label{lemma:n}
Suppose $\E(r(x))=0$ in the homotopy Gaussian. Then $r(x)$ can be written uniquely as 
\[
r(x)= \sum_{i=1}^N a_i (x^i-E(x^i))\] where $N$ is the degree of $r$.
\end{lemma}
\begin{proof}
The element $r(x)$ is closed under $d$. Since $\E$ is a quasi-isomorphism, $r(x)$ must be exact.
\end{proof}
This motivates the following definition.
\begin{definition}
Let $\E(r(x))=0$. Then $h(r)$ is defined as follows:
\[
h(r)=\sum_{i=1}^N a_i y_i
\]
so that $d(h(r))=r(x)$.
\end{definition}
\begin{definition}
Define an $L_\infty$ homotopy $\Lambda$ from $(\C,0)$ to $(V,d)$ as follows:
\[
\Lambda_n(1\odot\cdots\odot 1)=p^n + t(q^n-p^n) + h(p^n-q^n)dt
\]
The components are evidently closed and because there are no higher brackets in $(V,d)$ this is an $L_\infty$ homotopy. Evaluation reveals that it is an $L_\infty$ homotopy between $\varphi X$ and $\varphi Y$. We will shorten $\Lambda_n(1\odot\cdots\odot 1)$ to $\Lambda_n$.
\end{definition}
To get an $L_\infty$ homotopy between $X$ and $Y$ it is now only necessary to compose with $\varphi^{-1}$.
\begin{construction}
Define the following collection of linear maps from $S^n\C$ to $V$.
\[
H_n(1\odot\cdots \odot 1)=\sum_{k=1}^n\sum_{P_k(n)}(-1)^{k-1}(k-1)!\Lambda_{p_1}\cdots \Lambda_{p_k}
\]
Here $P_k(n)$ denotes partitions of $n$ into $k$ parts of size $p_1,\ldots, p_k$.
\end{construction}
\begin{proposition}
The previous construction is an $L_\infty$ homotopy from $(\C,0)$ to $(V,d^\varphi)$ between $X$ and $Y$.
\end{proposition}
\begin{proof}
Lemmas \ref{lemma:1} through \ref{lemma:n}, along with a quick computation of $\varphi^{-1}$, yield the result.
\end{proof}
\begin{remark}
If two collections of homotopy random variables are homotopic, that means that those collections are indistinguishable by cumulants. However, this does not mean that the collections will remain indistinguishable if they are extended to larger collections of homotopy random variables. For instance, in the one variable Gaussian, we can define homotopy random variables $X$ and $Y$ where $X(1)=x$ and $Y(1)=-x$ with all higher maps zero. 
These homotopy random variables are homotopic and indistinguishable. 

Define collections of two homotopy random variables $\overline{X}$ and $\overline{Y}$ where 
\begin{align*}
\overline{X}((1,0))&=x& \overline{Y}((1,0))&=-x
\\
\overline{X}((0,1))&=1& \overline{Y}((0,1))&=1
\end{align*}with all higher maps zero. Then $\overline{X}$ and $\overline{Y}$ are not homotopic. Indeed, the second cumulant distinguishes them.
\end{remark}

\bibliographystyle{plain} 
\bibliography{HPT-ref.bib}

\end{document}